\documentclass[12pt,a4paper,reqno]{amsart}
\usepackage{amsmath, amsthm, amsfonts, amssymb}
\usepackage{hyperref}
\usepackage{enumitem}
\usepackage{tabularx}

\numberwithin{equation}{section}

\theoremstyle{plain}
\newtheorem{theorem}[subsection]{Theorem}
\newtheorem{proposition}[subsection]{Proposition}
\newtheorem*{prop-no-number}{Proposition}
\newtheorem*{prop-}{Proposition}
\newtheorem{lemma}[subsection]{Lemma}
\newtheorem{corollary}[subsection]{Corollary}

\theoremstyle{definition}
\newtheorem{defn}[subsection]{Definition}

\newtheorem{question}[subsection]{Question}

\newtheorem{remark}[subsection]{Remark}

\renewcommand{\leq}{\leqslant}
\renewcommand{\geq}{\geqslant}
\renewcommand{\subset}{\subseteq}
\renewcommand{\supset}{\supseteq}

\def\E{\mathbb{E}}
\def\Z{\mathbb{Z}}

\def\C{\mathbb{C}}
\def\N{\mathbb{N}}

\def\F{\mathbb{F}}

\def\Ghat{\widehat{G}}
\def\fhat{\widehat{f}}

\DeclareMathOperator{\Bohr}{Bohr}
\DeclareMathOperator{\Spec}{Spec}

\newcommand{\Zmod}[1]{\Z_{#1}} % \Z/#1\Z or \Z_{#1} -- this is so that one can easily change from Z_p to Z/pZ depending on the journal's style

\providecommand{\abs}[1]{\lvert#1\rvert}
\providecommand{\norm}[1]{\lVert #1 \rVert}

\providecommand{\ceiling}[1]{\lceil#1\rceil}

\theoremstyle{plain}
\makeatletter
\newtheorem*{rep@theorem}{\rep@title}
\newcommand{\newreptheorem}[2]{%
\newenvironment{rep#1}[1]{%
\def\rep@title{#2 \ref*{##1}}%
\begin{rep@theorem}}%
{\end{rep@theorem}}}
\makeatother

\newreptheorem{theorem}{Theorem}

\title[Sparse sets]{Roth's theorem for four variables and additive structures in sums of sparse sets}

\author{Tomasz Schoen}
\address{Faculty of Mathematics and Computer Science\\
Adam Mickiewicz University\\
Umultowska 87, 61-614 Pozna\'n, Poland}
\email{schoen@amu.edu.pl}

\author{Olof Sisask}
\address{Department of Mathematics\\
     KTH\\
     100 44 Stockholm\\
     Sweden}
\email{sisask@kth.se}


\begin{document}

\begin{abstract}
We show that if $A \subset \{1,\ldots,N\}$ does not contain any solutions to the equation $x+y+z=3w$ with the variables not all equal, then
	\[ \abs{A} \leq \frac{N}{ \exp\left(c (\log N)^{1/7}\right) }, \]
	where $c > 0$ is some absolute constant. In view of Behrend's construction, this bound is of the right shape: the exponent $1/7$ cannot be replaced by any constant larger than $1/2$.

	We also establish a related result, which says that sumsets $A+A+A$ contain long arithmetic progressions if $A \subset \{1,\ldots,N\}$, or high-dimensional subspaces if $A \subset \F_q^n$, even if $A$ has density of the shape above.
\end{abstract}

\maketitle

\parindent 0mm
\parskip   4mm

\section{Introduction}

This paper is concerned with two types of problems in additive combinatorics, namely solving linear equations in subsets of abelian groups and finding additive structures in sumsets, with a focus on being able to deal with relatively sparse sets. We discuss these in turn, focusing on the historically most important case of sets of integers.

\subsection*{Roth-type results}

Roth's famous theorem on arithmetic progressions says that if a set $A \subset [N] := \{1,\ldots,N\}$ does not contain any non-trivial three-term arithmetic progressions, that is solutions to the equation $x+z=2y$ with $x,y,z$ not equal, then $\abs{A} = o(N)$. In fact:

\begin{theorem}[Roth's theorem \cite{roth}]Let $r_3(N)$ denote the largest size of a subset of $[N]$ with no non-trivial three-term progressions. Then, for $N$ large enough,
	\[ r_3(N) = O\left(\frac{N}{\log\log N}\right). \]
\end{theorem}

This theorem has been central to additive combinatorics, and improving the above bound has been the object of much research and has led to a wealth of interesting techniques being developed; see for example \cite{szemeredi, heath-brown, bourgain:roth, bourgain:roth2, sanders:roth2, sanders:roth, bloom:roth}, to which we also refer for more history on the problem. It is not yet known, however, whether $r_3(N) = O(N/\log N)$; the current best upper bounds, due to Sanders \cite{sanders:roth} and Bloom \cite{bloom:roth}, give 
\[ r_3(N) \leq \frac{N}{(\log N)^{1-o(1)} }. \]

By contrast, the best lower bound on $r_3(N)$, coming from constructions of large subsets of $[N]$ with no non-trivial progressions, is of the form
	\[ r_3(N) \geq \frac{N}{\exp\big(O(\log N)^{1/2}\big)} \]
and is due to Behrend \cite{behrend} (but see also \cite{elkin, green-wolf}).

Now, most proofs of Roth's theorem easily extend to provide similar upper bounds for any translation invariant equation
\begin{equation}\label{eqn:gen_eqn}
c_1 x_1 + \cdots + c_k x_k = 0 \quad \text{where $k \geq 3$, $c_j \in \Z\setminus\{0\}$, and $c_1+\cdots+c_k = 0$}\footnote{This is the translation invariance property.},
\end{equation}
and Behrend's argument extends directly to any such equation with one negative coefficient and the rest positive, that is of the form $a_1 x_1 + \cdots + a_l x_l = b y$ with the $a_j$ positive integers summing to $b$. Furthermore, a somewhat folklore philosophy was that whatever techniques worked for additive combinatorial problems involving three variables would also work for those involving four or more, and vice versa, with the bounds being similar. The work \cite{sanders:bogolyubov} of Sanders led to this being questioned in the context of sumsets, however, and the first-named author and Shkredov \cite{schoen-shkredov} subsequently showed that much stronger bounds than those given for $r_3(N)$ above hold for equations in six or more variables. A representative example:

\begin{theorem}[\cite{schoen-shkredov}]
	Suppose $A \subset [N]$ does not contain any solutions to $x_1+\cdots+x_5 = 5y$ in distinct integers. Then
	\[ \abs{A} \leq \frac{N}{ \exp\big(\Omega(\log N)^{1/7}\big) }. \]
\end{theorem}

Here one has an almost matching lower bound: Behrend's construction gives sets $A$ of size at least $\exp\big(-O(\log N)^{1/2}\big)N$ that do not contain any solutions to this equation.

Around the same time, Bloom \cite{bloom:4} established improved bounds for four and five variable equations, inspired by Sanders's technique from \cite{sanders:roth}:

\begin{theorem}[\cite{bloom:4}]
	Suppose $A \subset [N]$ does not contain any non-trivial\footnote{A solution $(x_1,\ldots,x_k)$ to \eqref{eqn:gen_eqn} is called trivial if one can partition the index set $[k]$ into parts on which the variables $x_j$ are constant and the coefficients $c_j$ sum to $0$. For example $(x,\ldots,x)$.} solutions to the equation in \eqref{eqn:gen_eqn}. Then
	\[ \abs{A} \leq \frac{N}{ (\log N)^{k-2-o_c(1)} }. \]
\end{theorem}

There thus remained an almost exponential gap between the lower and upper bounds for four and five variable equations. In this paper we show that one indeed has Behrend-shape upper bounds for these. For example:

\begin{theorem}\label{thm:4varInts}
	Suppose $A \subset [N]$ does not contain any non-trivial solutions to the equation $x+y+z = 3w$. Then
	\[ \abs{A} \leq \frac{N}{ \exp\big(\Omega(\log N)^{1/7}\big) }. \]
\end{theorem}

In the much-studied finite field setting, where $[N]$ is replaced by a vector space over a finite field, we establish the following slightly stronger result.

\begin{theorem}\label{thm:finFieldRoth}
	Let $q$ be a prime power and let $A \subset \F_q^n$ be a set of size $\alpha \, q^n$. If $A$ does not contain any non-trivial solutions to $x+y+z=3w$, then
	\[ \alpha \leq \exp\left( -\Omega\left( n^{1/5} \right) \right). \]
\end{theorem}

By contrast, the best bound known for three-term progressions in this setting comes from the intricate work of Bateman and Katz \cite{bateman-katz}, who showed that if $A \subset \F_3^n$ is free of non-trivial three-term progressions then $\abs{A} \leq 3^n/n^{1+\epsilon}$, where $\epsilon$ is some strictly positive constant.

Before we move on, let us make a quick remark about our arguments. These are somewhat different to those of \cite{schoen-shkredov}, which used the bounds of Sanders \cite{sanders:bogolyubov} for a result known as the Bogolyubov--Ruzsa lemma. However, the proof of this lemma used in turn an almost-periodicity result of Croot and the second-named author \cite{CS}, and this will together with an insight from \cite{sanders:bogolyubov} be of key importance in our proofs. This is actually part of the motivation behind this paper: while one aim is to prove strong bounds for as close a problem as possible to Roth's theorem, another is to attempt to illustrate the natural limitations of the ideas of \cite{CS, sanders:bogolyubov}. We thus give two different proofs of Theorem \ref{thm:finFieldRoth} that demonstrate different aspects of the results; see Section \ref{section:finFields}.

\subsection*{Structures in sumsets}
Another big direction of additive combinatorics is to study the structure of sumsets $A+B = \{a+b : a \in A, b \in B\}$ for various types of sets $A$ and $B$ in an abelian group. Here we focus on the case of three-fold sumsets $3A := A+A+A$ where $A$ is a large subset of $[N]$ or a finite abelian group $G$, as was first tackled by Freiman, Halberstam and Ruzsa \cite{FHR}. Suppose $A \subset [N]$ has size at least $\alpha N$, $\alpha > 0$. The following lower bounds for the length of an arithmetic progression in $3A$ are known\footnote{Here and throughout the paper we let $c$ and $C$ denote unspecified positive absolute constants whose values need not be the same at different occurrences.}.

%\begin{tabular}{r l l}
%\end{tabular}
\newcommand{\hs}{\hspace{2em}}
\begin{tabularx}{\textwidth}{l X l}
 Density range & \hs Length of AP in $3A$ & \\
$\alpha \geq (\log N)^{-1/3 + o(1)}$		&	\hs $N^{c \alpha^3}$		&	F--H--R \cite{FHR} \\
$\alpha \geq (\log N)^{-1/2 + o(1)}$	&	\hs $N^{c \alpha^{2+o(1)}}$	&	Green \cite{green:longAPs} \\
$\alpha \geq (\log N)^{-1/2 + o(1)}$	&	\hs $N^{c \alpha}$		&	Sanders \cite{sanders:structInSumsets} \\
    $\alpha \geq (\log N)^{-1+o(1)}$	&	\hs $N^{c \alpha^{1+o(1)}}$		&	Henriot \cite{henriot}  \\
$\alpha \geq (\log N)^{-2+o(1)}$	&	\hs $\exp\left( (\alpha^{1/2+o(1)} \log N)^{1/2} \right)$	&	Henriot \cite{henriot} 
\end{tabularx}

Henriot \cite{henriot} gives a useful and clear summary of the history of the problem, and we refer there for more information. Let us also mention that Henriot's results are actually more powerful in the asymmetric case of sumsets $A+B+C$: in this setup \cite{henriot} allows $B$ and $C$ to be much sparser, namely of densities around $\exp\left(-O(\log N)^c \right)$, as long as the density of $A$ is more-or-less as above.

Here we prove the following, which is non-trivial in the range $\alpha \geq \exp\left(- c (\log N)^{1/5} \right)$.

\begin{theorem}\label{thm:3Aints}
	Let $A \subset [N]$ have size at least $\alpha N$. Then $3A$ contains an arithmetic progression of length at least
	\[ \alpha \exp\left( \left( \frac{c \log N}{\log^3(2/\alpha)} \right)^{1/2} \right). \]
\end{theorem}

The length of the progression here is of course much smaller than previous results for large densities, being on par with what is known for just $A+A$ in this case, but when $\alpha$ gets small enough this theorem applies whereas those above do not. Let us mention, however, that there is a combinatorial argument due to Croot, Ruzsa and the first-named author \cite{CRS} that guarantees arithmetic progressions in $2A := A+A$ of length around $c\log N/\log(2/\alpha)$, which certainly extends the non-trivial density range further albeit with fairly short progressions. We thus know of quite different behaviours for different densities, but a lack of examples pervades. The best example we know of comes from \cite{FHR}: there it is shown that, for any $\alpha < c$, there is a set $A \subset [N]$ of size at least $\alpha N$ for which $3A$ does not contain an arithmetic progression of length
\begin{equation} N^{2/\log(1/\alpha)}. \label{eqn:FHRexample} \end{equation}
Theorem \ref{thm:3Aints} thus gives an answer of the right shape $\exp\left( (\log N)^c \right)$ for $\alpha = \exp\left(- (\log N)^c \right)$, but with a gap in the exponent on the $\log N$ compared to \eqref{eqn:FHRexample}.

These questions are also studied for subsets of vector spaces $\F_q^n$ over finite fields $\F_q$, where $q$ is considered fixed, but in this setting one generally looks at the dimensions of (affine) subspaces found in sumsets rather than arithmetic progressions, for obvious reasons. See for example \cite{green:finFields, sanders:structInSumsets, sanders:half, CLS} for more background. From the perspective of the present paper it is illuminating to consider what is known in this setting for $2A$, $3A$ and $4A$, for which the best bounds known for large densities are all due to Sanders. For $2A$, it is shown in \cite{sanders:half} that $2A$ contains an affine subspace of dimension at least $c \alpha n$ for $\alpha \geq C/n$. Sumsets $3A$ are known \cite{sanders:structInSumsets} to contain affine subspaces of dimension at least $n - C/\alpha$, and sumsets $4A$ are known \cite{sanders:bogolyubov} to contain affine subspaces of dimension at least $n - C \log^4(2/\alpha)$. Here we prove a result somewhat intermediate between the latter two:

\begin{theorem}\label{thm:3Avector}
	Let $A \subset \F_5^n$ be a set of size at least $\alpha \cdot 5^n$. Then $3A$ contains an affine subspace of dimension at least $c n/\log(2/\alpha)^3 - \log(1/\alpha)$.
\end{theorem}

We actually show somewhat more, namely that these three-fold sumsets contain lots of translates of the respective arithmetic progression or subspace; see Section \ref{section:3A} for further statements, and see also Section \ref{section:remarks} for some further comparisons of $2A$, $3A$ and $4A$.

The rest of this paper is structured as follows. In the next section we set up some notation and describe some preliminaries on density increments, convolutions and almost-periodicity. In Section \ref{section:finFields} we prove Theorem \ref{thm:finFieldRoth}; indeed, we give two proofs as already mentioned. We then proceed to a proof of the general case, starting with a review of Bohr sets in Section \ref{section:Bohr} and the development of the appropriate almost-periodicity results in Section \ref{section:ap}, and wrapping up with the density increment and iterative arguments in Sections \ref{section:densInc} and \ref{section:iteration}. We then turn to structures in $3A$ in Section \ref{section:3A}, and conclude with some remarks in Section \ref{section:remarks}.

\section{Notation and preliminaries}

If $A$ is a subset of a finite set $X$, we refer to $\mu(A) := \mu_X(A) := \abs{A}/\abs{X}$ as the \emph{density} of $A$ (in $X$).

\subsection*{The density increment strategy}
In proving the Roth-type theorems outlined above, we shall employ a so-called \emph{density increment} strategy, as have most proofs of Roth's theorem resulting in good bounds. This operates roughly as follows. Let $G$ be $[N]$ or $\F_q^n$. If $A \subset G$ has density $\alpha$ but contains no non-trivial solutions to $x+y+z=3w$, then one shows that $A$ has increased density $(1+c(\alpha))\alpha$ on a translate of some `large substructure' $V$ of $G$ -- say a long progression in the case of $[N]$ or a large subspace in the case of $\F_q^n$. Thus $\abs{A \cap (x+V)} \geq (1+c)\alpha \abs{V}$. One then looks at $(A-x) \cap V$, which is still solution-free by translation invariance, and tries to repeat the argument. One thus produces denser and denser solution-free sets on smaller and smaller substructures, but since a density can never increase beyond $1$, the iteration must at some point terminate (provided the function $c(\alpha)$ is nice enough). Generally this means that the substructures on which one is iterating must have become trivial, so as long as the original density is large enough for the increased densities to reach $1$ before the substructures become trivial, one has shown that the set must contain a solution to the equation.

Of course, all this is saying is roughly that we shall prove the result by induction; the whole game is to find arguments to make the substructures $V$ and the increments $c(\alpha)$ as large as possible, while keeping $V$ nice enough to iterate. For many proofs of Roth's theorem, the substructures on which one increments are directly related to the large Fourier coefficients of $A$; for us this is not quite the case, the substructures being uncovered instead by the probabilistic almost-periodicity results of \cite{CS}. To state one of these results in detail, let us introduce some further notation.

\subsection*{Normalisations, $L^p$-norms, convolutions}
Now, we have talked about densities above, and it is relatively standard practice in additive combinatorics these days to work with these rather than cardinalities of sets. An associated trend has been to furthermore use normalised convolutions and $L^p$ norms. In this paper we shall find it useful to work with both densities and cardinalities, as we shall operate relatively `locally' later on. We thus speak of densities, but write, for an abelian group $G$, a subset $X \subset G$, a function $f : G \to \C$ and a real number $p \geq 1$,
        \[ \mu_X := 1_X/\abs{X}, \qquad \norm{f}_p^p := \sum_{x \in G}\abs{f(x)}^p, \qquad f*g(x) := \sum_{y \in G} f(y)g(x-y). \]
	Here and throughout, $1_X$ denotes the \emph{indicator function} of $X$, taking the value $1$ if its input lies in $X$ and $0$ otherwise.

	Convolutions really are central objects for us when pursuing a density increment strategy as outlined above. Indeed, the quantity $1_A*\mu_V(x)$ is precisely $\abs{A \cap (x-V)}/\abs{V}$, which is the relative density of $A$ on $x-V$, and the number of solutions to our equation is precisely $1_A*1_A*1_A*1_{-3\cdot A}(0)$. Crucially, however, we shall not prove our results by studying this function directly, as did most previous proofs, but we shall nevertheless deal with similar convolutions, and for this the key tools will be certain almost-periodicity results.

\subsection*{Almost-periodicity}
Our main tool for showing properties of convolutions is the following $L^p$-almost-periodicity result, which is a version of the main theorem of \cite{CS}, but with somewhat less detailed moment estimates in the probabilistic arguments; see for example \cite{sanders:bogolyubov, CLS} for a proof.

\begin{theorem}\label{thm:Lp-ap}
	Let $p \geq 2$, $\epsilon \in (0,1)$ and $k \in \N$ be parameters. Let $A, L, S$ be finite subsets of an abelian group. Suppose $\abs{A+S} \leq K\abs{A}$. Then there is a set $T \subset S$ with $\abs{T} \geq 0.99 K^{-Cpk^2/\epsilon^2} \abs{S}$ such that
	\[ \norm{ 1_A*1_L(\cdot+t) - 1_A*1_L }_p \leq \epsilon \abs{A} \abs{L}^{1/p} \quad\text{for all $t \in kT-kT$.} \]
\end{theorem}

The result thus says that, for two sets $A$ and $L$, provided $A$ is structured in the sense of not growing much under addition with some set $S$, one can find lots of $L^p$-almost-periods of the convolution $1_A*1_L$, these being elements $t$ for which this function does not change by much (in $L^p$) upon translation by $t$.

We shall bootstrap this to other variants later on; let us now instead turn to the model setting, in which we can simply quote a similar result.

\section{Two proofs in the finite field setting}\label{section:finFields}

Here we shall prove Theorem \ref{thm:finFieldRoth}, which said that if a subset $A \subset \F_q^n$ of density $\alpha$ does not contain any non-trivial solutions to the equation
\begin{equation} x+y+z=3w \label{eqn:main} \end{equation}
then $\alpha \leq \exp\left(-c n^{1/5} \right)$. Note that, for this equation, a solution is trivial if and only if $x=y=z=w$, and that the result is trivial if $q$ is divisible by $2$ or $3$, so we assume throughout that it is not. We shall actually give two different proofs of this result, one more analytic and one more combinatorial -- but both following the density increment strategy outlined in the previous section. It turns out that the former proof extends more easily to the setting of more general finite abelian groups, from which one can deduce Theorem \ref{thm:4varInts}, whereas the latter serves as inspiration for the later proofs finding structures in sumsets.

In both proofs we shall use the following bootstrapped version of Theorem \ref{thm:Lp-ap}, which is a specialisation of \cite[Theorem 7.4]{CLS}.

\begin{theorem}\label{thm:Lp}
Let $p \geq 2$ and $\epsilon \in (0,1)$. Let $G = \F_q^n$ be a vector space over a finite field 
and suppose $A, L \subset G$ have $\mu(A) \geq \alpha$. Then there is a subspace $V$ of codimension
\[ d \leq Cp \epsilon^{-2} \log(2/\epsilon \alpha)^2 \log(2/\alpha) \]
such that, for each $t \in V$,
\[ \norm{ 1_A*1_L(\cdot+t) - 1_A*1_L }_{L^p} \leq \epsilon \abs{A} \abs{L}^{1/p}. \]
\end{theorem}

\subsection*{First proof: via $L^\infty$-almost-periodicity of three-fold convolutions}\hfill

This proof is based on the fact that if $A$ does not contain any non-trivial solutions to \eqref{eqn:main}, then $1_{-3\cdot A}*1_A*1_{A+A}(0) = \abs{A}$, which is very small. Two-fold convolutions like this are, however, fairly continuous functions: we shall deduce from a certain almost-periodicity result that $1_{-3\cdot A}*1_A*1_{A+A}*\mu_V(0)$ is then also small for some large subspace $V$. If $\abs{A+A}$ is large, a simple averaging then implies that $A$ has a density increment on a translate of $V$. If, on the other hand, $\abs{A+A}$ is small, then one is done by a similar, if slightly simpler, argument.

The relevant almost-periodicity result is the following, but note that this will be superseded by a slightly more efficient and general version in Section \ref{section:ap}.

\begin{theorem}\label{thm:Linfty-ap_FF}
	Let $\epsilon \in (0,1)$ and let $A, M, L \subset \F_q^n$ have $\mu(A), \mu(M) \geq \alpha$. Then there is a subspace $V$ of codimension at most $C \epsilon^{-2} \log(2/\epsilon \alpha)^2 \log (2/\alpha)^2$ such that
	\[ \abs{ 1_A*1_M*1_L(x+t) -  1_A*1_M*1_L(x) } \leq \epsilon \abs{A} \abs{M} \]
	for all $x \in G$ and $t \in V$. In particular
	\[ \abs{ 1_A*1_M*1_L(0) - 1_A*1_M*1_L*\mu_V(0) } \leq \epsilon \abs{A} \abs{M}. \]
\end{theorem}
\begin{proof}
	Apply Theorem \ref{thm:Lp} with $p = C\log(2/\alpha)$ and $\epsilon/2$ to get a subspace $V$ of the required codimension such that
	\[ \norm{1_A*1_L(\cdot+t) - 1_A*1_L }_p \leq \tfrac{1}{2} \epsilon \abs{A} \abs{L}^{1/p}. \]
	Then, for $r$ with $1/r + 1/p = 1$, H\"older's inequality gives
\begin{align*}
\norm{ 1_A*1_M*1_L(\cdot + t) - 1_A*1_M*1_L }_\infty &\leq \norm{ 1_M }_r \norm{1_A*1_L(\cdot+t) - 1_A*1_L }_p \\
							    &\leq \tfrac{1}{2}\epsilon \abs{A} \abs{M} (\abs{L}/\abs{M})^{1/p}, \end{align*}
	whence the first claim is proved. The second follows from the triangle inequality.
\end{proof}

We now split into two cases, depending on whether the sumset $A+A$ is large or not.

\subsubsection*{Large sumset}
In the large sumset case, where $\mu(A+A) \geq \tfrac{1}{2}$, we shall make use of the fact that $1_{-3\cdot A}*1_A*1_{A+A}(0) = \abs{A}$ if $A$ is free from solutions to \eqref{eqn:main}, which means that the convolution $1_{-3\cdot A}*1_A*1_{(A+A)^c}$ takes a really large value. Though perhaps not clear in this formulation, this argument was inspired by those of \cite{ernie-me:Roth, ernie-me:Bogolyubov-Roth}.

\begin{proposition}\label{prop:noSolsLargeSumsetImpliesInc}
	Let $A \subset \F_q^n$ have density $\alpha$ and size at least $8$. Suppose $\mu(A+A) \geq \tfrac{1}{2}$ and that $A$ does not contain any non-trivial solutions to \eqref{eqn:main}. Then there is a subspace $V$ of codimension at most $C \log(2/\alpha)^4$ such that $1_A*\mu_V(x) \geq \frac{3}{2} \alpha$ for some $x$.
\end{proposition}

Recall that $1_A*\mu_V(x) = \abs{A \cap (x-V)}/\abs{V}$, and so the conclusion says that $A$ has massively increased density on some affine subspace of low codimension.

\begin{proof}
	Apply Theorem \ref{thm:Linfty-ap_FF} with $M=-3\cdot A$, $L=A+A$ and $\epsilon = 1/8$ to get a subspace $V$ of the required codimension such that
	\[ 1_{-3 \cdot A}*1_A*1_{A+A}*\mu_V(0) \leq 1_{-3\cdot A}*1_A*1_{A+A}(0) + \tfrac{1}{8}\abs{A}^2. \]
	Since $1_{-3\cdot A}*1_A*1_{A+A}(0) = \abs{A} \leq \abs{A}^2/8$, we thus have 
	\[ 1_{-3 \cdot A}*1_A*1_{A+A}*\mu_V(0) \leq \tfrac{1}{4} \abs{A}^2, \]
	and so
	\[ 1_{-3 \cdot A}*1_A*1_{(A+A)^c}*\mu_V(0) \geq \tfrac{3}{4} \abs{A}^2. \]
	The left-hand side here is at most $\abs{A} \abs{(A+A)^c} \norm{1_A*\mu_V}_\infty$, and so we are done.
\end{proof}

\subsubsection*{Small sumset}
That one can obtain a good density increment for $A$ when $A+A$ is small is well known, and a result almost sufficing for our purposes is contained in \cite{sanders:bogolyubov} -- see e.g. Theorem 9.1 there. We shall however use the following.

\begin{proposition}\label{prop:sumsetImpliesInc}
Suppose $A \subset \F_q^n$ has density $\alpha$ and $\mu(A+A) \leq \frac{1}{2}$. Then there is a subspace $V$ of codimension at most $C (\log 2/\alpha)^4$ such that $1_A*\mu_V(x) \geq \frac{3}{2} \alpha$ for some $x$.
\end{proposition}
\begin{proof}
	Apply Theorem \ref{thm:Linfty-ap_FF} with $M=A$, $L=-(A+A)$ and $\epsilon = 1/4$ to get a subspace $V$ of the required codimension such that
	\[ \abs{ 1_A*1_A*1_{-(A+A)}(0) - 1_A*1_A*1_{-(A+A)}*\mu_V(0) } \leq \tfrac{1}{4} \abs{A}^2. \]
	But $1_A*1_A*1_{-(A+A)}(0) = \abs{A}^2$ since $1_A*1_A$ is supported on $A+A$, and so
	\[ 1_A*1_A*1_{-(A+A)}*\mu_V(0) \geq \tfrac{3}{4}\abs{A}^2. \]
	Since the left-hand side here is at most $\abs{A} \abs{A+A} \norm{1_A*\mu_V}_\infty$, the result follows.
\end{proof}

Note that we did not need to assume that $A$ was free of solutions to any equations here.

\subsubsection*{Completing the proof: iterating}

Combining these propositions, one immediately obtains
\begin{corollary}\label{cor:densIncFF}
	Let $A \subset \F_q^n$ have density $\alpha$ and size at least $8$. Suppose $A$ does not contain any non-trivial solutions to \eqref{eqn:main}. Then there is a subspace of codimension at most $C \log(2/\alpha)^4$ such that $1_A*\mu_V(x) \geq \frac{3}{2} \alpha$ for some $x$.
\end{corollary}

We now simply iterate this corollary to complete the proof.

\begin{proof}[Proof of Theorem \ref{thm:finFieldRoth}]
	If $A \subset G := \F_q^n$ has density $\alpha$, size at least $8$ and is free of non-trivial solutions to \eqref{eqn:main}, then Corollary \ref{cor:densIncFF} gives us a subspace $V \leq G$ of codimension at most $C \log(2/\alpha)^4$ and an element $x \in G$ for which
	\[ \abs{(A-x) \cap V} \geq \tfrac{3}{2}\alpha \abs{V}, \]
	that is, a subspace in which $A-x$ has density at least $\tfrac{3}{2}\alpha$. Note that $A-x$ is still free of non-trivial solutions to \eqref{eqn:main} by translation invariance. We then repeat this argument with $G$ replaced by $V$, and so on, obtaining solution-free sets of increasing densities $\alpha_j$ in spaces of lowering dimension $n_j$, with $\alpha_1 = \alpha$ and $n_1 = n$. Assuming $\alpha_j \geq 8 q^{-n_j}$ at each stage, we thus have
	        \[ n_{j+1} \geq n_j - C \log(2/\alpha)^4 \geq n - C j \log(2/\alpha)^4, \]
		and
		\[ \alpha_{j+1} \geq \tfrac{3}{2} \alpha_j \geq \left(\tfrac{3}{2}\right)^j \alpha. \]
		Since the density cannot increase beyond $1$, this process must terminate with some $j \leq C \log(2/\alpha)$. If the claimed bound $\alpha \leq \exp\left(-c n^{1/5} \right)$ does not hold then we have $n_j \geq n - C  j \log(2/\alpha)^4 \geq n/2$ by the time of termination, and so running out of dimensions is not a reason for the process to terminate. Thus we must have $\alpha_j < 8 q^{-n_j}$. But this is easily seen to imply the claimed bound anyway, and we are done.
\end{proof}

Before we go on to give our second proof, let us make a quick remark about the types of solutions we have considered.

\begin{remark}
In the statement of Theorem \ref{thm:finFieldRoth} we forbade all non-trivial solutions to \eqref{eqn:main} in $A$, these being any non-constant quadruples $(x,y,z,w)$ for which $x+y+z=3w$. This has the effect of forbidding $A$ from containing solutions to certain other equations as well, such as non-trivial three-term arithmetic progressions -- if $x,y,z$ are distinct and lie in arithmetic progression, then the quadruple $(x,y,z,y)$ solves our equation. Though we did not pursue this issue above for the sake of clarity of exposition, let us mention that incorporating a short additional argument in fact shows that the same bound holds if one only disallows solutions where all the variables are distinct, so that one is only disallowing solutions to this equation and not any `sub-equations'. 
\end{remark}

\subsection*{Second proof: via properties of three-fold sumsets}\hfill

The following property of three-fold sumsets encodes the key to this proof.

\begin{proposition}\label{prop:BAA}
	Let $\eta \in (0,1)$ and let $A, B \subset \F_q^n$ be sets of densities $\alpha$, $\beta$ respectively. Then there is a subspace $V$ of codimension at most $C\log(2/\eta \beta) \log(2/\alpha)^3$ and a set $X \subset B$ with $\abs{X} \geq 0.99 \abs{B}$ such that
	\[ \abs{(x+V) \cap (B+A-A)} \geq (1-\eta)\abs{V} \]
	for every $x \in X$.
\end{proposition}
Another way of putting the conclusion is that $1_{B+A-A}*\mu_V(x) \geq 1-\eta$ for each $x \in X$.
\begin{proof}
	Apply Theorem \ref{thm:Lp} with $p = C\log(2/\eta \beta)$, $\epsilon = 1/2$ and $L = B-A$ to get a subspace $V$ of the required codimension such that
	\[ \norm{ 1_A*1_{B-A}(\cdot + t) - 1_A*1_{B-A} }_p \leq \tfrac{1}{2} \abs{A} \abs{B-A}^{1/p} \]
	for each $t \in V$.

	Let $X$ consist of all $x \in B$ such that $\abs{(x+V) \cap (B+A-A)} \geq (1-\eta)\abs{V}$, so that if $x \notin X$ then $1_A*1_{B-A}(x+t) = 0$ for more than $\eta \abs{V}$ elements $t \in V$. Then
	\[ \eta \abs{V} \sum_{x \in B\setminus X} 1_A*1_{B-A}(x)^p < \sum_{t \in V} \norm{ 1_A*1_{B-A}(\cdot + t) - 1_A*1_{B-A} }_p^p \leq \tfrac{1}{2^p} \abs{A}^p \abs{B-A} \abs{V}. \]
	But $1_A*1_{B-A}(x) = \abs{A}$ for each $x \in B$, and so this implies that
	\[ \abs{B \setminus X} < \tfrac{1}{2^p} \eta^{-1} \abs{B-A} \leq 0.01 \abs{B}, \]
	which completes the proof.
\end{proof}

\begin{corollary}\label{cor:ABC}
	Let $\eta \in (0,1)$ and let $A, B, C \subset \F_q^n$ have $\mu(A), \mu(C) \geq \alpha$ and $\mu(B) \geq \beta$. Then there is a subspace $V$ of codimension at most $C\log(2/\eta \beta) \log(2/\alpha)^3$, an element $t \in \F_q^n$ and a set $X \subset B+t$ with $\abs{X} \geq 0.99 \abs{B}$ such that
	\[ \abs{(x+V) \cap (A+B+C)} \geq (1-\eta)\abs{V} \]
	for every $x \in X$.
\end{corollary}
\begin{proof}
	Since $\sum_t 1_A*1_C(t) = \abs{A}\abs{C}$, there is some $t$ such that $\mu(A \cap (t-C)) \geq \alpha^2$. Applying Proposition \ref{prop:BAA} with this intersection instead of $A$ completes the proof.
\end{proof}

Using this we give a second proof of Corollary \ref{cor:densIncFF}, finding a good density increment.

\begin{proof}[Second proof of Corollary \ref{cor:densIncFF}]
	Partition $A = A_1 \cup A_2$ with $\abs{A_1} = \ceiling{ \frac{4}{5} \abs{A} }$ and apply Corollary \ref{cor:ABC} with $\eta := \alpha/2$, $B = C = -A_1$ and $3\cdot A_2$ in place of $A$. This gives us a subspace $V$ of codimension at most $C\log(2/\alpha)^4$, an element $t$ and a set $X \subset t-A$ with $\abs{X} \geq \frac{3}{4} \abs{A}$ such that
	\[ \abs{(x+V) \cap (3\cdot A_2 - A_1 - A_1)} \geq (1-\eta)\abs{V} \quad \text{for each $x \in X$.} \]
	Since $A$ does not contain any non-trivial solutions to \eqref{eqn:main}, $A$ and $3\cdot A_2 - A_1 - A_1$ are disjoint, whence
\begin{equation}
	\abs{(x+V) \cap A} \leq \tfrac{1}{2} \alpha \abs{V} \quad \text{for each $x \in X$.} \label{eqn:smallInt}
\end{equation}
	Since $V$ is a subspace, this in fact holds for all $x \in X+V$. How large is this sumset? Well, if $1_X*\mu_V(x) \geq \tfrac{3}{2} \alpha$ for some $x$, then we would have a density increment of the kind we are after, so let us assume that $1_X*\mu_V(x) < \tfrac{3}{2}\alpha$ for all $x$. Then
	\[ \abs{X} = \sum_{x \in X+V} 1_X*\mu_V(x) < \tfrac{3}{2}\alpha \abs{X+V}, \]
	and so \eqref{eqn:smallInt} holds for at least $\abs{X+V} \geq \tfrac{1}{2} \abs{G}$ elements $x$. In other words, $1_A*\mu_V(x) \leq \frac{1}{2}\alpha$ for at least half of the elements of the group. Since the average of this function over the whole group is $\alpha$, we must have $1_A*\mu_V(x) \geq \frac{3}{2}\alpha$ for some $x$, and so we are done.
\end{proof}

Since Theorem \ref{thm:finFieldRoth} followed directly from Corollary \ref{cor:densIncFF}, this completes the proof.

\subsection*{Extending the arguments}
Both of these proofs of Theorem \ref{thm:finFieldRoth} can be extended to handle the case of sets of integers using the machinery of regular Bohr sets pioneered by Bourgain \cite{bourgain:roth}, each with their own sets of difficulties. It turns out this process is more straightforward for the first proof, however, and so it is this that we shall present, starting in the next section with a review of the basic theory surrounding Bohr sets. The second proof is however very much related to the proofs we shall give for the results on structures in sums of sparse sets, as should become apparent.

\section{Bohr sets and their elementary properties}\label{section:Bohr}
When one wants to perform a density increment argument of the type we have just used in groups without a rich subgroup structure, it is by now rather established practice to turn to Bohr sets as a natural substitute for subspaces. In an abelian group $G$, we define these in terms of the dual group $\Ghat$ of characters, consisting of homomorphisms from $G$ to $\C^\times$ with the group operation given by pointwise multiplication.

\begin{defn}
	Let $\Gamma \subset \Ghat$ and let $\rho \geq 0$. We define the \emph{Bohr set} on these data by
	\[ \Bohr(\Gamma, \rho) = \{ x \in G : \abs{\gamma(x) - 1} \leq \rho \text{ for all $\gamma \in \Gamma$}\}. \]
	We refer to $\abs{\Gamma}$ as the \emph{rank} of the Bohr set, and $\rho$ as its \emph{radius}\footnote{Note that these quantities are not well-defined in terms of just the set itself, but we think of these data as being included in the definition of the Bohr set.}. We say that $\Bohr(\Gamma', \rho') \leq \Bohr(\Gamma, \rho)$ is a \emph{sub-Bohr set} if $\Gamma' \supset \Gamma$ and $\rho' \leq \rho$; note in particular that this implies containment as sets. We shall frequently need to scale the radii of our Bohr sets: if $B = \Bohr(\Gamma, \rho)$ and $\delta \geq 0$, then we write $B_\delta = \Bohr(\Gamma, \delta \rho)$.
\end{defn}

We refer the reader to Section 4.4 in the book \cite{TV} of Tao and Vu for the proofs of the following lemmas and for more background\footnote{The constants appearing here are somewhat different to those in \cite{TV}, as we have defined Bohr sets in terms of quantities of the form $\abs{z-1}$ rather than $\arg(z)$.}.

\begin{lemma}\label{lemma:BohrGrowth}
	Let $\Gamma \subset \Ghat$ be a set of $d$ characters, let $\rho \in [0,2]$, and let $B = \Bohr(\Gamma, \rho)$. Then we have the size estimate
	\[ \abs{B} \geq (\rho/2\pi)^d \abs{G} \]
	the doubling estimate
	\[ \abs{B_2} \leq 6^d \abs{B}, \]
	and, for $\delta \in [0,1]$, the decay estimate
	\[ \abs{B_\delta} \geq (\delta/2)^{3d} \abs{B}. \]
\end{lemma}

In particular $\abs{B+B} \leq 6^d \abs{B}$, since $B_\delta + B_\epsilon \subset B_{\delta+\epsilon}$ by the triangle inequality. Thus Bohr sets have fairly small doubling if $d$ is small. Subspaces, however, enjoy the stronger property that $\abs{V+V} = \abs{V}$ regardless of dimension, and this discrepancy in doubling constants reflects an underlying issue that means our argument becomes terribly inefficient if we simply try to replace subspaces with Bohr sets. In giving a proof of Roth's theorem with strong bounds, Bourgain \cite{bourgain:roth} showed how to work around this issue, namely by working with pairs of Bohr sets $(B,B_\delta)$ with $\delta$ small, for which $\abs{B+B_\delta} \leq \abs{B_{1+\delta}}$. A priori this need not be close to $\abs{B}$, but the following property ensures this.

\begin{defn}[Regularity]
	We say that a Bohr set $B$ of rank $d$ is regular if 
	\[ 1 - 12 d \abs{\delta} \leq \frac{\abs{B_{1+\delta}}}{\abs{B}} \leq 1 + 12 d \abs{\delta} \]
	whenever $\abs{\delta} \leq 1/12d$.%All of these 12s could be made 11, or 10.5, but 12 looks nicer, right?
\end{defn}

The constant $12$ here is of course not particularly important, but we include it for definiteness. Now, not all Bohr sets are regular, but it is a consequence of the doubling estimate $\abs{B} \leq 6^d \abs{B_{1/2}}$ that growth must be somewhat limited around some slight rescaling of $B$:

\begin{lemma}\label{lemma:BohrReg}
If $B$ is a Bohr set, then there is a $\delta \in [\frac{1}{2}, 1]$ for which $B_\delta$ is regular.
\end{lemma}

%This is a consequence of the estimate on $\abs{B_2}$ above: if there was significant growth around $B_\delta$ for each $\delta \in [1/2,1]$, then by combining these one would violate the doubling estimate $\abs{B} \leq 6^d \abs{B_{1/2}}$.

If $B$ is regular of rank $d$ we have the useful property that $\abs{B+B_\delta} \leq 2\abs{B}$ whenever $\delta \leq 1/12d$. We also have the following useful consequence of regularity, resting simply on an application of the triangle inequality.

\begin{lemma}\label{lemma:regConv}
	If $B$ is a regular Bohr set of rank $d$ and $B' \subset B_\delta$ with $\delta \leq \epsilon/24d$, then
	\[ \norm{ \mu_B*\mu_{B'} - \mu_B }_{L^1(G)} \leq \epsilon. \]
\end{lemma}

We require finally an arithmetic property of Bohr sets, which follows from the size estimate in Lemma \ref{lemma:BohrGrowth} and the inclusion $k B_{1/k} = B_{1/k} + \cdots + B_{1/k} \subset B$.

\begin{lemma}\label{lemma:BohrAP}
	Let $N$ be a prime and let $B \subset \Zmod{N}$ be a Bohr set of rank $d \geq 1$ and radius $\rho \in [0,2]$. Then $B$ contains an arithmetic progression of size at least $\frac{1}{2\pi} \rho N^{1/d}$.
\end{lemma}

\section{$L^\infty$-almost-periodicity relative to Bohr sets}\label{section:ap}

To carry out the strategy of Section \ref{section:finFields} with Bohr sets in place of groups, the first thing we need to do is prove an appropriate analogue of Theorem \ref{thm:Linfty-ap_FF}. Of course, not only do we need to replace the subspace $V$ in the conclusion with a Bohr set -- which is entirely straightforward -- but we are only allowed to assume density in a Bohr set rather than in a group. It turns out that this is also fairly straightforwardly achievable.

\subsection*{Almost-periodicity with dense sets}

Recall Theorem \ref{thm:Lp-ap}, the $L^p$-almost-periodicity result for two-fold convolutions. From this we argue straightforwardly as with Theorem \ref{thm:Linfty-ap_FF} to obtain the following $L^\infty$-almost-periodicity result for three-fold convolutions.

\begin{theorem}\label{thm:Linfty-ap}
	Let $\epsilon \in (0,1)$ and $k \in \N$ be parameters. Let $A, M, L, S$ be finite subsets of an abelian group. Suppose $\abs{A+S} \leq K\abs{A}$ and $\eta := \abs{M}/\abs{L} \leq 1$. Then there is a set $T \subset S$ with $\abs{T} \geq \exp\left(-C k^2 \epsilon^{-2} \log(2/\eta) \log(2K) \right) \abs{S}$ such that
	\[ \norm{ 1_A*1_M*1_L(\cdot+t) - 1_A*1_M*1_L }_\infty \leq \epsilon \abs{A} \abs{M} \quad\text{for all $t \in kT-kT$.} \]
\end{theorem}
\begin{proof}
	Apply Theorem \ref{thm:Lp-ap} with parameters $\epsilon/2$ and $p$ to be specified to obtain a set $T$ of almost-periods for $1_A*1_L$. By H\"older's inequality we then have, for $\frac{1}{p}+\frac{1}{q} = 1$ and any $t \in kT-kT$,
	\begin{align*}
		\norm{ 1_A*1_M*1_L(\cdot+t) - 1_A*1_M*1_L }_\infty &\leq \norm{1_M}_q \norm{ 1_A*1_L(\cdot+t) - 1_A*1_L }_p \\
		&\leq \tfrac{1}{2} \epsilon \abs{A} \abs{M} (\abs{L}/\abs{M})^{1/p}.
	\end{align*}
	Picking $p = 3 \log(2\abs{L}/\abs{M})$ yields the result.
\end{proof}

\begin{remark}
	Note that the set $T$ one obtains does not in fact depend on $M$ but only on $\abs{M}/\abs{L}$. Also, since the methods of \cite{CS} worked for non-abelian groups, a version of the above result holds for arbitrary groups, and one could also replace $1_M$ and $1_L$ by functions more general than indicator functions, but we shall only apply it in the above case.
\end{remark}

Finally we shall bootstrap this to find not only a large set of translates, but a \emph{structured} set: a Bohr set of translates. The price we shall pay is that we shall need to assume that the set $A$ interacts nicely with a Bohr set and not just an arbitrary set $S$. The main idea of the proof is to couple Theorem \ref{thm:Linfty-ap} with Chang's theorem on the structure of large spectra, which was one of the main insights that led to the powerful results \cite{sanders:bogolyubov} of Sanders. To state this properly we shall need the Fourier transform; the results of the following subsection are the only ones in this paper that appeal to Fourier analysis.

\subsection*{Almost-periodicity with Bohr sets}
For a function $f : G \to \C$ on a finite abelian group $G$ we define the Fourier transform $\fhat : \Ghat \to \C$ on the dual group $\Ghat$ by
	\[ \fhat(\gamma) := \sum_{x \in G} f(x) \overline{\gamma(x)}. \]
Writing $\E_{x \in X} = \abs{X}^{-1} \sum_{x \in X}$, the Fourier inversion formula, Parseval's identity and the convolution identity then take the form
	\begin{gather*}
	f(x) = \E_{\gamma \in \Ghat}\, \fhat(\gamma) \gamma(x), \\
	\sum_{x \in G} \abs{f(x)}^2 = \E_{\gamma \in \Ghat}\, \abs{\fhat(\gamma)}^2, \text{ and}\\
	\widehat{f*g}(\gamma) = \fhat(\gamma) \widehat{g}(\gamma).
	\end{gather*}
Finally, for a set $X \subset G$, write
	\[ \Spec_\delta(\mu_X) := \{ \gamma \in \Ghat : \abs{\widehat{\mu_X}(\gamma)} \geq \delta \} \]
	for the \emph{$\delta$-large spectrum} of $\mu_X = 1_X/\abs{X}$. See \cite{TV} for more on all of this.

Chang's theorem \cite[Lemma 4.36]{TV} says that the large spectrum $\Spec_\delta(\mu_X)$ is `low-dimensional': it is contained in the $\{-1,0,1\}$-span of a set of at most $C \delta^{-2} \log(1/\mu_G(X))$ characters. An immediate and useful consequence is that all the characters in $\Spec_\delta(\mu_X)$ can be approximately annihilated by a low-rank Bohr set if $X$ is relatively dense in $G$. Sanders proved an efficient version of such a consequence when $X$ is a dense subset of a Bohr set rather than the group; the following is \cite[Proposition 4.2]{sanders:structInSumsets}.%, but we could equivalently couple Lemmas 6.2 and 6.3 in the later \cite{sanders:bogolyubov} for our purposes.

\begin{proposition}[Chang--Sanders]\label{prop:localChang}
Let $\delta, \nu \in (0, 1]$. Let $G$ be a finite abelian group, let $B = \Bohr(\Gamma, \rho) \subset G$ be a regular Bohr set of rank $d$ and let $X \subset B$. Then there is a set of characters $\Lambda \in \Ghat$ and a radius $\rho'$ with 
\[ \abs{\Lambda} \ll \delta^{-2} \log(2/\mu_B(X)) \quad \text{and} \quad \rho' \gg \rho \nu \delta^2/d^2\log(2/\mu_B(X)) \]
such that
\[ \abs{1 - \gamma(t)} \leq \nu \quad \text{for all $\gamma \in \Spec_\delta(\mu_X)$ and $t \in \Bohr(\Gamma \cup \Lambda, \rho')$}. \]
\end{proposition}

The aforementioned bootstrapping can now take place via a standard argument.

\begin{theorem}[$L^\infty$-almost-periodicity with Bohr sets]\label{thm:Linfty-ap_Bohr}
	Let $\epsilon \in (0,1)$. Let $A, M, L$ be subsets of a finite abelian group $G$, and let $B \subset G$ be a regular Bohr set of rank $d$ and radius $\rho$. Suppose $\abs{A+S} \leq K\abs{A}$ for a subset $S \subset B$ with $\mu_B(S) \geq \sigma > 0$, and assume $\eta := \abs{M}/\abs{L} \leq 1$. Then there is a regular Bohr set $B' \leq B$ of rank at most $d + d'$ and radius at least $\rho \epsilon \eta^{1/2}/d^2d'$, where
	\[ d' \ll \epsilon^{-2} \log^2(2/\epsilon\eta) \log(2/\eta) \log(2K) + \log(1/\sigma), \]
such that	
	\[ \norm{ 1_A*1_M*1_L(\cdot+t) - 1_A*1_M*1_L }_\infty \leq \epsilon \abs{A} \abs{M} \quad\text{for all $t \in B'$.} \]
In particular,
	\[ \norm{ 1_A*1_M*1_L*\mu_{B'} - 1_A*1_M*1_L }_\infty \leq \epsilon \abs{A} \abs{M}. \]
\end{theorem}
\begin{proof}
	Begin by applying Theorem \ref{thm:Linfty-ap} to $1_A*1_M*1_L$ with parameters $\epsilon$ and $k := \ceiling{C\log(2/\epsilon \eta)}$ to obtain a set $T \subset S$ with
	\[ \mu_B(T) \geq \exp\left( -C \epsilon^{-2} k^2 \log(2/\eta) \log(2K)\right) \sigma \]
	such that
	\[ \norm{ 1_A*1_M*1_L(\cdot+t) - 1_A*1_M*1_L }_\infty \leq \epsilon \abs{A} \abs{M} \quad\text{for all $t \in kT-kT$.} \]
	Fix some $z \in T$ and set $X = T-z$, so that the above inequality holds for all $t \in kX$. Thus, by the triangle inequality,
	\[ \norm{ 1_A*1_M*1_L*\mu_X^{(k)} - 1_A*1_M*1_L }_\infty \leq \epsilon \abs{A} \abs{M}, \]
	where $\mu_X^{(k)} := \mu_X*\cdots*\mu_X$ with $k$ copies of $\mu_X$. It thus suffices to establish the theorem with $1_A*1_M*1_L*\mu_X^{(k)}$ in place of $1_A*1_M*1_L$, and so we switch now to this.
	
	Noting that translating $X$ does not affect the conclusion of Proposition \ref{prop:localChang}, apply this proposition to $T = X + z$ with parameters $\delta = 1/2$ and $\nu = \epsilon \eta^{1/2}$ together with Lemma \ref{lemma:BohrReg} to get a regular Bohr set $B' \leq B$ of the required rank and radius such that
	\[ \abs{1 - \gamma(t)} \leq \epsilon \eta^{1/2} \quad \text{for all $\gamma \in \Spec_{1/2}(\mu_X)$ and $t \in B'$.} \]
	For any $x \in G$ and $t \in B'$ we then have, by the Fourier inversion formula, triangle inequality and convolution identity,
	\begin{align}
		\abs{1_A*1_M*1_L*\mu_X^{(k)}(x+t) &- 1_A*1_M*1_L*\mu_X^{(k)}(x)} \nonumber\\ \leq &\ \E_{\gamma \in \Ghat}\, \abs{\widehat{1_A}(\gamma)} \abs{\widehat{1_M}(\gamma)} \abs{\widehat{1_L}(\gamma)} \abs{\widehat{\mu_X}(\gamma)}^k \abs{\gamma(t) - 1}. \label{eqn:FourierBound}
	\end{align}
	For each term in this average, consider whether $\gamma \in \Spec_{1/2}(\mu_X)$ or not. If $\gamma \in \Spec_{1/2}(\mu_X)$ we have $\abs{\gamma(t) - 1} \leq \epsilon \eta^{1/2}$, and if not then $\abs{\widehat{\mu_X}(\gamma)}^k \leq 1/2^k \leq \epsilon \eta^{1/2}$. Thus \eqref{eqn:FourierBound} is at most twice
	\[ \epsilon \eta^{1/2}\ \E_{\gamma \in \Ghat}\, \abs{\widehat{1_A}(\gamma)} \abs{\widehat{1_M}(\gamma)} \abs{\widehat{1_L}(\gamma)}. \]
	Using the trivial inequality $\abs{\widehat{1_A}(\gamma)} \leq \abs{A}$ and Cauchy-Schwarz plus Parseval \`a la
	\[ \E_{\gamma \in \Ghat}\, \abs{\widehat{1_M}(\gamma)} \abs{\widehat{1_L}(\gamma)} \leq \left( \E_{\gamma \in \Ghat}\, \abs{\widehat{1_M}(\gamma)}^2 \right)^{1/2} \left( \E_{\gamma \in \Ghat}\, \abs{\widehat{1_L}(\gamma)}^2 \right)^{1/2} = \abs{M}^{1/2}\abs{L}^{1/2} \]
	finishes the proof, after replacing $\epsilon$ with $\epsilon/4$.
\end{proof}

\begin{remark}
	The regime in which the above argument is set up to be efficient is one in which $A$ is thought of as extremely small, but structured in the sense of not expanding much under addition to a Bohr set, $M$ as being of `medium' size and $L$ as being large.
	The main utility of this result over previous Fourier-analytic ones of this sort, then, stems from the fact that the dependence on $\abs{L}/\abs{M}$ in the rank of $B'$ is only polylogarithmic rather than polynomial.
\end{remark}

\section{Obtaining density increments on Bohr sets}\label{section:densInc}

The following proposition drives the density increment argument.

\begin{proposition}\label{prop:dens_inc}
	Let $G$ be a finite abelian group of order not divisible by $3$, let $B \subset G$ be a regular Bohr set of rank $d$ and radius $\rho$, and let $A \subset B$ have relative density $\mu_B(A) \geq \alpha$. Assume that $\abs{B} \geq (Cd/\alpha)^{3d}$. If $A$ does not contain any non-trivial solutions to $x+y+z=3w$, then $A$ has relative density at least $\tfrac{5}{4} \alpha$ on a translate of a Bohr set $B' \leq B$ of rank at most $d + d'$ and radius at least $\rho \alpha^{3/2}/d^5d'$, where $d' \ll \log(2/\alpha)^4$.
\end{proposition}

As in the model case, we prove this differently in two cases depending on whether a particular sumset is large or not. In each case we make the further assumption that our given set $A$ is dense also in a narrower sub-Bohr set.

\subsection*{The large sumset, solution-free case}

\begin{lemma}\label{lemma:large_sumset}
	Let $G$ be a finite abelian group of order not divisible by $3$, let $B \subset G$ be a regular Bohr set of rank $d$ and radius $\rho$, and let $A \subset B$ have relative density $\mu_B(A) \geq \alpha$. Let $B' := B_\delta$ be a regular sub-Bohr set with $\delta := 1/Cd$ such that $\abs{B_{1+3\delta}} \leq 1.01 \abs{B}$, and assume that $A' := A \cap B'$ satisfies $\mu_{B'}(A') \geq \alpha$ and $\abs{A+A'} \geq \abs{A}/2\alpha$. If $\abs{A} \geq C$ and $A$ does not contain any non-trivial solutions to $x+y+z=3w$, then $\norm{1_A*\mu_{T}}_\infty \geq 1.8 \alpha$ for some Bohr set $T \leq B$ of rank at most $d + d'$ and radius at least $\rho \alpha^{1/2}/d^4 d'$, where $d' \ll \log^4(2/\alpha)$.
\end{lemma}
\begin{proof}
Define $S := 3\cdot B'_{\nu}$ where $\nu := 1/Cd$, so that (using the assumption on $\abs{G}$) $S$ is a Bohr set of rank $d$ and radius at least $\rho/Cd^2$, and note that, by regularity,
	\begin{equation}
		\abs{3\cdot A' + S} \leq \abs{B'_{(1 + \nu)}} \leq 2 \abs{B'} \leq \tfrac{2}{\alpha} \abs{3 \cdot A'}. \label{eqn:A'struct}
	\end{equation}
	Apply Theorem \ref{thm:Linfty-ap_Bohr} with $-3\cdot A'$ in place of $A$, $S$ as defined above, $M := A$, $L := B_{1+3\delta} \setminus (A+ A')$ and $\epsilon := \tfrac{1}{40}$. Our assumption $\abs{A+A'} \geq \abs{A}/2\alpha$ implies that 
\begin{equation}
	\abs{L} \leq 1.01 \abs{B} - \tfrac{1}{2\alpha} \abs{A} \leq \tfrac{0.501}{\alpha}\abs{A}, \label{eqn:Lbound}
\end{equation}
and so the parameter $\eta$ of that theorem is certainly at least $\alpha$. We may further take $K = 2/\alpha$ by \eqref{eqn:A'struct}, and so we get a Bohr set $T \leq S$ of rank at most $d + d'$ and radius at least $\rho \alpha^{1/2}/d^4 d'$, where $d' \leq C \log^4(2/\alpha)$, such that
	\[ \norm{ 1_{-3\cdot A'}*1_A*1_L*\mu_T - 1_{-3\cdot A'}*1_A*1_L }_\infty \leq \tfrac{1}{40} \abs{A'} \abs{A}. \]
	Now, since $A$ does not contain any non-trivial solutions to $x+y+z=3w$, we have
	\[ 1_{-3\cdot A'}*1_A*1_{A+A'}(0) = \abs{A'}. \]
	Thus
	\begin{align*}
		1_{-3\cdot A'}*1_A*1_L*\mu_T(0) &\geq 1_{-3\cdot A'}*1_A*(1_{B_{1+3\delta}} - 1_{A+A'})(0) - \tfrac{1}{40} \abs{A'} \abs{A} \\
		&= \tfrac{39}{40}\abs{A'}\abs{A} - \abs{A'} \\
		&\geq \tfrac{19}{20} \abs{A'}\abs{A}, 
	\end{align*}
	provided $\abs{A} \geq 40$. By the pigeonhole principle, then, there must be some element $x$ for which
	\[ 1_A*\mu_T(x) \geq \tfrac{19}{20} \abs{A}/\abs{L} \geq 1.8 \alpha, \]
	by \eqref{eqn:Lbound}.
\end{proof}

\subsection*{The small sumset case}

Again, the case in which $A+A'$ is small can be handled a slightly simpler fashion.

\begin{lemma}\label{lemma:small_sumset}
	Let $A \subset G$, let $B \subset G$ be a regular Bohr set of rank $d$ and radius $\rho$, and let $A' \subset B$ have relative density $\mu_B(A') \geq \alpha$. If $\abs{A+A'} \leq \abs{A}/2\alpha$, then $\norm{1_A*\mu_{T}}_\infty \geq 1.8 \alpha$ for some Bohr set $T \leq B$ of rank at most $d + d'$ and radius at least $\rho \alpha^{1/2}/d^3 d'$, where $d' \ll \log^4(2/\alpha)$.
\end{lemma}
\begin{proof}
	Let $S = B_\nu$ where $\nu := 1/Cd$, so that
	\[ \abs{A' + S} \leq \abs{B_{1+\nu}} \leq \tfrac{2}{\alpha} \abs{A'}. \]
	Applying Theorem \ref{thm:Linfty-ap_Bohr} with $A'$ in place of $A$, this set $S$, $M := A$, $L := -A-A'$ and $\epsilon := \tfrac{1}{10}$, we may take $\eta \geq 2\alpha$ and $K := 2/\alpha$ to get a Bohr set $T \leq S$ of rank at most $d + d'$ and radius at least $\rho \alpha^{1/2}/d^3 d'$ where $d' \leq C\log^4(2/\alpha)$ such that
	\[ \norm{ 1_{A'}*1_A*1_{-A-A'}*\mu_T - 1_{A'}*1_A*1_{-A-A'} }_\infty \leq \tfrac{1}{10} \abs{A'} \abs{A}. \]
	Now, $1_A*1_{-A-A'}(x) = \abs{A}$ for any $x \in -A'$, and so $1_{A'}*1_A*1_{-A-A'}(0) = \abs{A'} \abs{A}$. Thus
	\begin{align*}
		1_{A'}*1_A*1_{-A-A'}*\mu_T(0) \geq \tfrac{9}{10}\abs{A'}\abs{A}.
	\end{align*}
	Pigeonholing and using the assumption on $\abs{A+A'}$, there is thus some $x$ for which
	\[ 1_A*\mu_T(x) \geq \tfrac{9}{10} \abs{A}/\abs{A+A'} \geq 1.8 \alpha. \qedhere \]
\end{proof}

\subsection*{Rescaling and putting the cases together}

We need one final tool in order to put the previous two lemmas together to prove Proposition \ref{prop:dens_inc}, namely a simple averaging argument due to Bourgain \cite{bourgain:roth} that in practice allows us to assume that a dense subset $A$ of a Bohr set $B$ is also large on a sub-Bohr set $B_\delta$ for some not-too-small $\delta$.

\begin{lemma}\label{lemma:TwoScales}
	Let $B$ be a regular Bohr set of rank $d$, let $A \subset B$ have relative density $\alpha$, and let $B', B'' \subset B_\delta$ where $\delta \leq \alpha/Cd$. Then either 
	\begin{enumerate}
		\item there is an $x \in B$ such that $1_A*\mu_{B'}(x) \geq \tfrac{7}{10}\alpha$ and $1_A*\mu_{B''}(x) \geq \tfrac{7}{10}\alpha$, or
		\item $\norm{1_A*\mu_{B'}}_\infty \geq \tfrac{5}{4}\alpha$ or $\norm{1_A*\mu_{B''}}_\infty \geq \tfrac{5}{4}\alpha$.
	\end{enumerate}
\end{lemma}
\begin{proof}
	Since $B$ is regular, picking the constant $C$ large enough yields
	\[ \abs{ 1_A*\mu_B*\mu_{B'}(0) - 1_A*\mu_B(0) } \leq \norm{\mu_B*\mu_{B'} - \mu_B}_1 \leq \tfrac{1}{40} \alpha \]
	by Lemma \ref{lemma:regConv}, and similarly for $B''$. Since $1_A*\mu_B(0) = \mu_B(A) = \alpha$, this implies that 
	\[ \E_{x \in B}\, \big(1_A*\mu_{B'}(x) + 1_A*\mu_{B''}(x) \big) \geq (2 - \tfrac{1}{20}) \alpha, \]
	and so there exists $x \in B$ such that $1_A*\mu_{B'}(x) + 1_A*\mu_{B''}(x) \geq (2 - \frac{1}{20}) \alpha$. Fix such an $x$. If we are not in the second case of the conclusion, we then have  
	\[ 1_A*\mu_{B'}(x) \geq (2- \tfrac{1}{20})\alpha - \tfrac{5}{4}\alpha = \tfrac{7}{10}\alpha, \]
	and similarly for $B''$, and so we are done. 
\end{proof}

\begin{proof}[Proof of Proposition \ref{prop:dens_inc}]
	We start by rescaling our Bohr set so that $A$ is large at two scales simultaneously: apply Lemma \ref{lemma:TwoScales} with $\delta := \alpha/Cd$ picked so that $B' := B_\delta$ is regular, and with $B'' := B'_{\delta'}$ where $\delta' := 1/Cd$ is picked so that this is regular and $\abs{B'_{1+3\delta'}} \leq 1.01\abs{B'}$. If we are in the second case of the conclusion of that lemma, then we have a density increment on a translate of a Bohr set of rank $d$ and radius at least $\alpha/Cd^2$, in which case we are done. So assume instead that we get an element $x \in B$ such that
	\[ 1_A*\mu_{B'}(x),\, 1_A*\mu_{B''}(x) \geq \tfrac{7}{10}\alpha, \]
	and let $A' := (A-x)\cap B'$, $A'' := (A-x)\cap B''$; these sets thus have relative densities at least $\alpha' := \tfrac{7}{10}\alpha$ in their respective Bohr sets. Note by translation invariance that $A'$ also does not contain any non-trivial solutions to our equation.

	Now, if $\abs{A'+A''} \leq \abs{A'}/2\alpha'$, then we apply Lemma \ref{lemma:small_sumset} with $(A',A'',B'')$ in place of $(A,A',B)$ to get that
	\[ \norm{1_A*\mu_T}_\infty \geq \norm{1_{A'}*\mu_T}_\infty \geq 1.8 \alpha' \geq \tfrac{5}{4} \alpha, \]
	where $T$ is a Bohr set of rank at most $d+d'$ and radius at least $\rho \alpha^{3/2}/d^5 d'$, with $d' \ll \log^4(2/\alpha)$, and so we are done.

	If, on the other hand, $\abs{A'+A''} \geq \abs{A'}/2\alpha'$, then we apply Lemma \ref{lemma:large_sumset} with $(A',A'',B')$ in place of $(A,A',B)$ to get precisely the same conclusion, provided that $\abs{A'} \geq C$. A quick computation using Lemma \ref{lemma:BohrGrowth} shows that this is ensured by our assumption that $\abs{B} \geq (Cd/\alpha)^{3d}$, and so we are done.
\end{proof}

\section{The iterative argument}\label{section:iteration}

We now iterate the density increment result of the preceding section to prove our theorem.

\begin{theorem}\label{thm:RothAbelianGroup}
	Let $G$ be a finite abelian group of order $N$ not divisible by $3$. If $A \subset G$ does not contain any non-trivial solutions to $x+y+z=3w$, then
	\[ \abs{A} \leq \frac{ N }{ \exp\left( c (\log N)^{1/7} \right) }. \]
\end{theorem}
\begin{proof}
	Initialise $A_1 = A$, $B^{(1)} = \Bohr(\{1\},2) = G$, $d_1 = 1$, $\rho_1 = 2$ and $\alpha_1 = \alpha = \abs{A}/\abs{G}$.
	We run the following iterative scheme until the condition required for doing so fails.

	If $\abs{B^{(j)}} \geq (C d_j/\alpha_j)^{3d_j}$, then we apply Proposition \ref{prop:dens_inc} to our sets and parameters to produce a new Bohr set $B^{(j+1)} \leq B^{(j)}$ of rank $d_j$ and radius $\rho_j$ satisfying
\begin{gather*}
	d_{j+1} \leq d_j + C \log^4(2/\alpha_j) \leq C j \log^4(2/\alpha), \\
	\rho_{j+1} \geq \rho_j \alpha_j^{3/2}/C d_j^5 \log^4(2/\alpha)
\end{gather*}
	and a set $A_{j+1} = (A_j-x_j) \cap B^{(j+1)} \subset B^{(j+1)}$ (for some $x_j$) of relative density
	\[ \alpha_{j+1} \geq \tfrac{5}{4}\alpha_j \geq \left(\tfrac{5}{4}\right)^{j} \alpha. \]
	Note that $A_{j+1}$ has no non-trivial solutions to our equation by translation invariance.

	Since the density of a set can never increase beyond $1$, the growth of the $\alpha_j$ implies that we must no longer be able to iterate this process when $j = s$ for some $s \leq C \log(2/\alpha)$. Thus we must have $\abs{B^{(s)}} < (Cd_s/\alpha_s)^{3d_s}$. On the other hand, by Lemma \ref{lemma:BohrGrowth} we have $\abs{B^{(s)}} \geq (\rho_s/2\pi)^{d_s} \abs{G}$. Putting these together we certainly have
	\[ \abs{G} < \left( C d_s/\rho_s \alpha_s \right)^{3 d_s}. \]
	Now $d_s \leq C \log^5(2/\alpha)$, $\rho_s \geq (c \alpha)^{Cs}$ and $\alpha_s \geq \alpha$; putting these bounds in gives 
	\[ \abs{G} < \exp\left( C \log^7(2/\alpha) \right), \]
	which yields the bound of the theorem upon rearranging.
\end{proof}

\begin{remark}
	With minor modifications, one can of course also prove a version of this theorem with $A$ simply being dense in a Bohr set, rather than the full group; we omit the details.
\end{remark}

\section{Additive structures in sums of sparse sets}\label{section:3A}

We turn now to the questions of structures in sumsets, proving Theorems \ref{thm:3Aints} and \ref{thm:3Avector}. This will be somewhat easier work than in the previous few sections as the arguments are iteration-free and so do not require the machinery associated with regular Bohr sets. We do, however, require the analogue of Theorem \ref{thm:Lp} for arbitrary finite abelian groups, this being another specialisation of \cite[Theorem 7.4]{CLS}:

\begin{theorem}\label{thm:Lp-ap_Bohr}
Let $p \geq 2$ and $\epsilon \in (0,1)$. Let $G$ be a finite abelian group and let $A, L \subset G$ be sets with $\mu(A) \geq \alpha$. Then there is a Bohr set $T$ of rank at most
\[ d := C p \epsilon^{-2} \log(2/\epsilon \alpha)^2\log (2/\alpha) \]
and radius at least $\epsilon \alpha^{1/2}/d$ such that, for each $t \in T$,
\[ \norm{ 1_A*1_L(\cdot+t) - 1_A*1_L }_p \leq \epsilon \abs{A} \abs{L}^{1/p}. \]	
\end{theorem}

Using this in place of Theorem \ref{thm:Lp}, the following can be proved in precisely the same way as Corollary \ref{cor:ABC}.

\begin{proposition}
Let $\eta \in (0,1)$ and let $A, B, C \subset G$ have densities $\alpha, \beta, \gamma$ respectively. Then there is a Bohr set $T \subset G$ of rank at most $d := C\log(2/\eta \beta) \log(2/\alpha\gamma)^3$ and radius at least $(\alpha\gamma)^{1/2}/d$, and an element $t \in G$, such that for any $V \subset T$ there is a set $X \subset B+t$ with $\abs{X} \geq 0.99 \abs{B}$ such that
%Let $\delta, \eta \in (0,1)$ and let $A, B, C \subset G$ have densities $\alpha, \beta, \gamma$ respectively. Then there is a Bohr set $T \subset G$ of rank at most $d := C\log(2/\eta \beta \delta) \log(2/\alpha\gamma)^3$ and radius at least $(\alpha\gamma)^{1/2}/d$, and an element $t \in G$, such that for any $V \subset T$ there is a set $X \subset B+t$ with $\abs{X} \geq (1-\delta) \abs{B}$ such that
	\[ \abs{(x+V) \cap (A+B+C)} \geq (1-\eta)\abs{V} \quad \text{for every $x \in X$.} \]
\end{proposition}

Note that if $C = -A$ then we can reduce the radius to $\alpha^{1/2}/d$ and take $t = 0$.

\begin{proposition}\label{prop:B+A-A_general}
	Let $A, B, C$ be sets of densities $\alpha, \beta, \gamma$ respectively in a finite abelian group $G$, and let $p \geq 1$. Then there is a Bohr set $T \subset G$ of rank at most $d := C p \left(\log 2/\alpha\gamma\right)^3$ and radius at least $(\alpha\gamma)^{1/2}/d$ such that, for any subset $V \subset T$ of size at most $\beta \cdot 2^p$, there is a set $X \subset B$ of size $\abs{X} \geq 0.99 \abs{B}$ such that a translate of $X+V$ is contained in $A+B+C$. 
\end{proposition}
\begin{proof}
	This follows immediately from the preceding proposition on taking $\eta = 1/(\beta\, 2^{p+1})$, so that $(1-\eta)\abs{V} > \abs{V} - 1$.
\end{proof}

One can also prove this directly from Theorem \ref{thm:Lp-ap_Bohr} following the proof of \cite[Theorem 1.4]{CLS} but taking into account the very large `higher energy' of $1_A*1_{B-A}$; this is, of course, very much related to the proof of Proposition \ref{prop:BAA}. 

We now have some easy corollaries. Theorem \ref{thm:3Aints} follows immediately from

\begin{theorem}\label{thm:ABCints}
	Let $A,B,C \subset [N]$ be sets of densities $\alpha, \beta, \gamma$. Then $A+B+C$ contains $X+P$ where $X \subset B$ has $\abs{X} \geq 0.99 \abs{B}$ and $P$ is an arithmetic progression of length at least
	\[ \exp\left( c \left(\frac{\log{N}}{\log^3(2/\alpha \gamma)} \right)^{1/2} - \log(1/\alpha \beta \gamma) \right). \]
\end{theorem}
\begin{proof}
By the standard trick of embedding $[N]$ into $\Zmod{N'}$ for $N'$ a prime between $6N$ and $12N$, it suffices to prove the statement with $[N]$ replaced by $\Zmod{N}$ for $N$ a prime, so we assume this setup instead.

Now apply Proposition \ref{prop:B+A-A_general} with $p = C \left(\frac{\log N}{\log^3(2/\alpha\gamma)} \right)^{1/2}$ to obtain a set $X \subset B$ and a Bohr set $T$ of rank $d \leq C p \log^3(2/\alpha\gamma)$ and radius at least $c \alpha\gamma/d$ satisfying that theorem's conclusion. By Lemma \ref{lemma:BohrAP}, $T$ contains an arithmetic progression of length at least $(c\alpha\gamma/d) N^{1/d}$. A quick calculation shows that the claimed arithmetic progression has length shorter than both this and $\beta\cdot 2^p$, whence we are done.
\end{proof}

Note that this result can be non-trivial for $\alpha$ and $\gamma$ as small as $\exp\left(-c (\log N)^{1/5} \right)$ and for $\beta$ even as small as $\exp\left(-c (\log N)^{1/2} \right)$. Also, since Bohr sets are extremely rich in additive structure, one can of course replace $P$ in the conclusion by other kinds of sets, such as generalised arithmetic progressions, which can then be much larger. Just measuring the length of a single progression, as we have done above, is nevertheless a simple and useful measure of the strength of the method.

In the finite field world we obtain the following generalisation of Theorem \ref{thm:3Avector}.

\begin{theorem}
	Let $A,B,C \subset \F_q^n$ be sets of densities $\alpha, \beta, \gamma$. Then $A+B+C$ contains $X+V$ where $X \subset B$ has $\abs{X} \geq 0.99\abs{B}$ and $V$ is an affine subspace of dimension at least 
	\[ \left( \frac{c n}{\log^3(2/\alpha \gamma)} - \log(1/\beta) \right)/\log q . \]
\end{theorem}
\begin{proof}
	This follows just as before: applying Proposition \ref{prop:B+A-A_general} with $p = c n/\log^3(2/\alpha\gamma)$, we obtain a large set $X \subset B$ and a subspace $T \leq \F_q^n$ of dimension at least $n - C p \log^3(2/\alpha\gamma)$ such that $A+B+C$ contains a translate of $X+V$ for any subset $V \subset T$ of size less than $\beta \cdot 2^p$. Noting that this is less than $\abs{T}$ and letting $V$ be a subspace of $T$ of size between $\beta \cdot 2^p/q$ and $\beta \cdot 2^p$ then does the job.
\end{proof}

Note that if $q = 5$, say, this can be non-trivial for $\alpha, \gamma$ as small as $\exp\left(- c n^{1/3} \right)$ and for $\beta$ as small as $5^{-cn}$ -- in other words, $\abs{B}$ can be as small as a power of $\abs{G}$ in this setup. One can also reach such densities in the $[N]$ world; see the next section.

\begin{remark}
	In the case that at $A$ or $C$ has very large density, the above results follow from those known for two-fold sumsets, with $X$ being the whole of $B$ even. The point here is thus that one can deal with much sparser sets, and the cost is only that one gets slightly fewer translates of the structure in $A+B+C$.
\end{remark}

\section{Concluding remarks}\label{section:remarks}

\subsection*{Other equations}
We only dealt with the equation $x+y+z=3w$ in this paper, but it should be clear that one can deal with a general translation invariant equation $c_1 x+ c_2 y + c_3 z + c_4 w = 0$ in precisely the same way, at least in the finite field setting. In the more general setting one needs to make some small alterations related to the radii of the Bohr sets involved, but as in the former case the main difficulty is notational. A similar remark applies to equations in five variables, where precisely the same bounds hold.

\subsection*{Lower bounds in finite fields}
What is the largest size of a subset of $\F_5^n$ with no non-trivial solutions to $x+y+z=3w$? Just as for three-term progressions, we do not know of a Behrend-type example in this setting; indeed the best we know of comes from taking products of examples for small $n$, resulting in sets of size around $\theta^n$ for some $\theta < 5$.

\subsection*{Small doubling instead of density}
Clearly one could work with small-sumset conditions instead of density conditions in many of the proofs in this paper, but there is not much incentive to do so in view of the nature of the bounds and the presence of effective `modelling lemmas' in the settings of interest; see for example \cite[Section 6]{green-ruzsa}.

\subsection*{Lower densities for the $A+B+C$ problem in the integers}
Theorem \ref{thm:ABCints} found arithmetic progressions in $A+B+C$ where one of the sets could have density as low as $\exp\left(-c (\log N)^{1/2} \right)$. To reach even lower densities, one can use the argument underlying \cite[Theorem 1.9]{CS}, again adding the idea of exploiting the higher energy of $1_A*1_{B-A}$:

\begin{theorem}
	Let $A,B,C \subset [N]$ be sets of densities $\alpha, \beta, \gamma$. Then $A+B+C$ contains an arithmetic progression of length at least
\[ \exp\left( c \left(\frac{\log{N}}{\log(2/\alpha \gamma)} \right)^{1/4} - \log(1/\beta) \right). \]
\end{theorem}

This is worse than the bound in Theorem \ref{thm:ABCints} for $\alpha = \beta = \gamma$, but for certain density combinations it actually wins out. For example, it allows one to take $\alpha$ and $\gamma$ to be as small as $N^{-c}$ provided $\beta$ is a constant. (Note, however, that in this particular range one is guaranteed constant-length progressions already in $A+C$, as follows from \cite{CRS}.) An answer to the following question would thus be interesting.

\begin{question}
	Suppose $A,B \subset [N]$ have densities $N^{-c}$ and $C \subset [N]$ has density $\exp\left(-C (\log N)^{2/3} \right)$. Must $A+B+C$ contain an arithmetic progression of length tending to infinity with $N$?
\end{question}

\subsection*{Correlations for $2A$, $3A$ and $4A$}
Following on from the discussion of subspaces in sumsets in the introduction, let us offer this perhaps illustrative comparison of results on correlations of $2A$, $3A$ and $4A$ with subspaces, where $A \subset \F_q^n$ has density $\alpha$.

\begin{itemize}[leftmargin=2em]
	\item $2A$ contains $1-\epsilon$ of an affine subspace of codimension at most $C \epsilon^{-2-o(1)} \log(1/\alpha)^4$.
	\item $3A$ contains $1-\epsilon$ of an affine subspace of codimension at most $C \log(1/\epsilon \alpha) \log(1/\alpha)^3$.
	\item $4A$ contains all of an affine subspace of codimension at most $C \log(1/\alpha)^4$.
\end{itemize}

The first and last bullets follow from Sanders's work \cite{sanders:bogolyubov} (and directly from Theorem \ref{thm:Linfty-ap_FF}), and the middle one from Proposition \ref{prop:BAA}. (Note also that the last bullet follows from the other two by inclusion-exclusion.) These results focus on the small density case: when $\alpha$ is large some prior results can offer better bounds; for example, for $\F_2^n$ Sanders showed in \cite{sanders:structInSumsets} that $2A$ contains $1-\epsilon$ of an affine subspace of codimension at most $C \alpha^{-2} \log(1/\epsilon)$, and codimension at most $C \alpha^{-1} \log(1/\epsilon)$ in \cite{sanders:half}.

It is, however, far from clear where the truth lies for these results -- not only in terms of the exponents on the logarithms but also in the qualitative differences between $3A$ and $4A$. It may very well be that the result for $4A$ actually holds for $3A$, as would have been expected prior to \cite{sanders:bogolyubov}, and any proof of this is likely to be useful in proving Behrend-shape bounds for Roth's theorem itself. On the other hand, any demonstrations of a genuine difference between $3A$ and $4A$, or three and four-variable equations, say, would also be very interesting.


\begin{thebibliography}{99}
	\bibitem{bateman-katz} M. Bateman and N. H. Katz, \emph{New bounds on cap sets,} J. Amer. Math. Soc. \textbf{25} (2012), no. 2, 585--613. \href{http://arxiv.org/abs/1101.5851}{arXiv:1101.5851}.
\bibitem{behrend} F. A. Behrend, \emph{On sets of integers which contain no three terms in arithmetical progression,} Proc. Nat. Acad. Sci. U. S. A. \textbf{32}, (1946). 331--332.
\bibitem{bloom:roth} T. F. Bloom, \emph{A quantitative improvement for Roth's theorem on arithmetic progressions,} \href{http://arxiv.org/abs/1405.5800}{arXiv:1405.5800}.
\bibitem{bloom:4} \bysame, \emph{Translation invariant equations and the method of Sanders,} Bull. Lond. Math. Soc. \textbf{44} (2012), no. 5, 1050--1067. \href{http://arxiv.org/abs/1107.1110}{arXiv:1107.1110}.
\bibitem{bourgain:roth} J. Bourgain, \emph{On triples in arithmetic progression,} Geom. Funct. Anal. \textbf{9} (1999), no. 5, 968--984.
\bibitem{bourgain:roth2} \bysame, \emph{Roth's theorem on progressions revisited,} J. Anal. Math. \textbf{104} (2008) 155--192.
\bibitem{CLS} E. Croot, I. \L aba and O. Sisask, \emph{Arithmetic progressions in sumsets and $L^p$-almost-periodicity,} Combin. Probab. Comput. \textbf{22} (2013), no. 3, 351--365. \href{http://arxiv.org/abs/1103.6000}{arXiv:1103.6000}.
\bibitem{CRS} E. Croot, I. Z. Ruzsa and T. Schoen, \emph{Arithmetic progressions in sparse sumsets,} Combinatorial number theory, 157--164, de Gruyter, Berlin, 2007.
\bibitem{ernie-me:Roth} E. Croot and O. Sisask, \emph{A new proof of Roth's theorem on arithmetic progressions,} Proc. Amer. Math. Soc. \textbf{137} (2009), 805--809. \href{http://arxiv.org/abs/0801.2577}{arXiv:0801.2577}.
\bibitem{CS} \bysame, \emph{A probabilistic technique for finding almost-periods of convolutions,} Geom. Funct. Anal. \textbf{20} (2010), no. 6, 1367--1396. \href{http://arxiv.org/abs/1003.2978}{arXiv:1003.2978}.
\bibitem{ernie-me:Bogolyubov-Roth} \bysame, \emph{Notes on proving Roth's theorem using Bogolyubov's method,} \url{http://people.math.gatech.edu/~ecroot/bogolyubov-roth2.pdf}
\bibitem{elkin} M. Elkin, \emph{An improved construction of progression-free sets,} Israel J. Math. \textbf{184} (2011), 93--128. \href{http://arxiv.org/abs/0801.4310}{arXiv:0801.4310}
\bibitem{FHR} G. A. Freiman, H. Halberstam, I. Z. Ruzsa, \emph{Integer sum sets containing long arithmetic progressions,} Jour. London Math. Soc. \textbf{46} (2), 193--201, 1992.
\bibitem{green:longAPs} B. Green, \emph{Arithmetic progressions in sumsets,} Geom. Funct. Anal. \textbf{12} (2002), no. 3, 584--597.
\bibitem{green:finFields} \bysame, \emph{Finite field models in additive combinatorics,} Surveys in combinatorics 2005, 1--27, London Math. Soc. Lecture Note Ser., 327, Cambridge Univ. Press, Cambridge, 2005. \href{http://arxiv.org/abs/math/0409420}{arXiv:math/0409420}.
\bibitem{green-ruzsa} B. Green and I. Z. Ruzsa, \emph{Freiman's theorem in an arbitrary abelian group,} J. Lond. Math. Soc. (2) \textbf{75} (2007), no. 1, 163--175. \href{http://arxiv.org/abs/math/0505198}{arXiv:math/0505198}.
\bibitem{green-wolf} B. Green and J. Wolf, \emph{A note on Elkin's improvement of Behrend's construction,} Additive number theory, 141--144, Springer, New York, 2010. \href{http://arxiv.org/abs/0810.0732}{arXiv:0810.0732}.
\bibitem{heath-brown} D. R. Heath-Brown, \emph{Integer sets containing no arithmetic progressions,} J. London Math. Soc. (2) \textbf{35} (1987), no. 3, 385--394.
\bibitem{henriot} K. Henriot, \emph{On arithmetic progressions in $A + B + C$,} Int. Math. Res. Notices, first published online June 11, 2013, doi:10.1093/imrn/rnt121. \href{http://arxiv.org/abs/1211.4917}{arXiv:1211.4917}.
\bibitem{roth} K. Roth, \emph{On certain sets of integers,} J. London Math. Soc. \textbf{28} (1953), 104--109.
\bibitem{ruzsa:linEqns1} I. Z. Ruzsa, \emph{Solving a linear equation in a set of integers, I,} Acta Arith. \textbf{65} (1993), 259--282.
\bibitem{sanders:structInSumsets} T. Sanders, \emph{Additive structures in sumsets,} Math. Proc. Cambridge Philos. Soc. \textbf{144} (2008), no. 2, 289--316. \href{http://arxiv.org/abs/math/0605520}{arXiv:math/0605520}.
\bibitem{sanders:half} \bysame, \emph{Green's sumset problem at density one half,} Acta Arith. \textbf{146} (2011), no. 1, 91--101. \href{http://arxiv.org/abs/1003.5649}{arXiv:1003.5649}.
\bibitem{sanders:roth2} \bysame, \emph{On certain other sets of integers,} J. Anal. Math. \textbf{116} (2012), 53--82. \href{http://arxiv.org/abs/1007.5444}{arXiv:1007.5444}.
\bibitem{sanders:roth} \bysame, \emph{On Roth's theorem on progressions,} Ann. of Math. \textbf{174} (2011), no. 1, 619--636. \href{http://arxiv.org/abs/1011.0104}{arXiv:1011.0104}.
\bibitem{sanders:bogolyubov} \bysame, \emph{On the Bogolyubov-Ruzsa lemma,} Anal. PDE \textbf{5} (2012), no. 3, 627--655. \href{http://arxiv.org/abs/1011.0107}{arXiv:1011.0107}.
\bibitem{schoen-shkredov} T. Schoen and I. Shkredov, \emph{Roth's theorem in many variables},  Israel J. Math. \textbf{199} (2014), no. 1, 287--308. \href{http://arxiv.org/abs/1106.1601}{arXiv:1106.1601}.
\bibitem{szemeredi} E. Szemer\'edi, \emph{Integer sets containing no arithmetic progressions,} Acta Math. Hungar. \textbf{56} (1990), no. 1--2, 155--158.
\bibitem{TV} T. Tao and V. H. Vu, \emph{Additive Combinatorics} (CUP, 2006).
\end{thebibliography}
\end{document}